\documentclass[a4paper,12pt]{amsart}
\usepackage{macros}

\title{Fourier--Mukai transforms for non-commutative complex tori}
\author[N.~Okuda]{Nobuki Okuda}
\address{
Graduate School of Mathematical Sciences,
The University of Tokyo,
3-8-1 Komaba,
Meguro-ku,
Tokyo,
153-8914,
Japan.}
\email{hiokuc8h18@gmail.com}
\date{}
\begin{document}

\begin{abstract}
Let $X$ be a complex torus of dimension $g$
and $\hat{X}$ be the dual torus.
For any $g(g-1)/2$-tuple $\lambda$
of complex numbers of absolute value $1$,
we define a non-commutative complex torus
$\sfX_\lambda$
as a sheaf of algebras on a real torus of dimension $g$.
We prove that
if all components of $\lambda$ are roots of unity,
then the category of coherent sheaves on $\sfX_\lambda$ is abelian
and derived-equivalent to the category of coherent sheaves
on $\hat{X}$
twisted by an element of the Brauer group of $\hat{X}$
determined by $\lambda$.
\end{abstract}

\maketitle{}

\section{Introduction}

Let $X$ and $\hat{X}$ be smooth projective varieties.
An integral functor
\begin{align}
\Phi \colon D^b(X) \to D^b(\hat{X})
\end{align}
between derived categories of coherent sheaves
is said to be a \emph{Fourier--Mukai functor}
if it is an equivalence.
A Fourier--Mukai functor $\Phi$
induces an isomorphism
\begin{align}
\HH(\Phi) \colon \HH^\bullet(X) \to \HH^\bullet(\hat{X})
\end{align}
of Hochschild cohomologies.
For any $u \in \HH^2(X),$
Toda
\cite{MR2477894}
gave a $\mathbb{C}[\varepsilon]/(\varepsilon^2)$-linear category
$D^b(X,u)$
of first order deformations of $X$ along $u$
and an equivalence
\begin{align}
\Phi^\dagger \colon D^b(X,u) \to D^b(\hat{X}, \HH(\Phi)(u))
\end{align}
extending $\Phi$.

The integral functor $\FM$ with the Poincaré line bundle as the integral kernel is the first example of a Fourier--Mukai functor 
given by
Mukai
\cite{MR607081}
for an abelian variety
$X$
and the dual abelian variety $\hat{X}.$
Under the Hochschild--Kostant--Rosenberg isomorphism
\begin{align}
\HH^2(X)
\cong H^0(\wedge^2 T_X)
\oplus
H^1(T_X)
\oplus 
H^2(\cO_X),
\end{align}
the induced map $\HH(\FM)$
sends 
$
H^0(\wedge^2 T_X)
$
to
$
H^2(\cO_{\hat{X}}). %TODO: この事実の証明または参考文献があったほうが良い。
$
This suggests that non-commutative deformations of abelian varieties
are Fourier--Mukai partners
of gerby deformations of dual abelian varieties.

While the notion of gerby deformations
is well-established in terms of twisted sheaves,
non-commutative deformations are much harder to define in general.
At the formal level,
the Moyal deformation quantizations give formal non-commutative deformations
of complex tori,
and derived equivalences to formal gerby deformations
are proved in \cite{MR2309993}.
The next problem is to construct non-commutative tori and derived equivalences
for non-formal parameters
(or even as a family over a complex manifold).

The first attempt to define non-commutative complex tori
is given by Schwarz
\cite{MR1865115},
who introduced the notion of complex structures
on non-commutative tori,
which are irrational rotation algebras
(see \cite{MR623572}, for example)
%TODO: 引用
regarded as non-commutative spaces
in the sense of Connes\cite{MR1303779}.
Categories of holomorphic vector bundles
on them
are studied in \cite{MR1977884}.

In the paper \cite{MR2648899} and the preprint \cite{block2006duality}, Block discussed Fourier--Mukai functors between non-commutative complex tori and gerby deformations of dual complex tori, using DG categories which include objects corresponding to quasi-coherent sheaves. 
In \cite{block2006duality}, he also announced to discuss coherent sheaves in a paper in preparation.

%%%%%%%%%%%%%%%
%Aim of paper
We now explain the results of this paper.
Let 
$
X = T / \Gamma
$
be a complex torus,
where
$
T \coloneqq (\bCx)^g
$
and
$\Gamma$ is a discrete subgroup of $T$
isomorphic to $\mathbb Z^g.$
The dual complex torus
$
\hat{X} \coloneqq \Pic^0 X
$
can naturally be identified with
$
\hat{\Gamma}/\hat{T}
$
where
$
\hat{\Gamma}
\coloneqq \Hom(\Gamma, \bCx)
\cong (\bCx)^g
$
and
$
\hat{T}
\coloneqq \Hom(T,\bCx)
\cong \bZ^g.
$
Let $\lambda\in H^2(\hat{T},U(1))\cong U(1)^{g(g-1)/2}$ be an element of the second group cohomology of $\hat{T}$ with values in $U(1).$
We construct a non-commutative deformation of $X$
with parameter $\lambda.$
When $\lambda$ takes values in roots of unity,
we give an equivalence
of the derived category of coherent sheaves
on the non-commutative deformation of $X$ with parameter $\lambda$
and
the derived category of coherent sheaves
on the gerby deformation of $\hat{X}$ with the same parameter $\lambda.$

Our construction of non-commutative complex tori is
different from those in \cite{MR1865115}, \cite{MR1977884} and \cite{block2006duality}.
Ours can be regarded as a patching of Archimedean analog of quantum analytic tori in \cite{MR2596639}.
To prove the equivalence of derived categories, we use the idea of equivariant Fourier--Mukai transforms developed in \cite{MR3039826}.

This paper is organized as follows:
In Section \ref{sc:q-Weyl},
we discuss $q$-Weyl algebras
as toy models
to motivate constructions in later sections.
In Section \ref{sc:commutative},
we recall Fourier--Mukai transforms for complex tori. 
In Section \ref{sc:Laurent},
we define non-commutative complex tori
by deforming sheaves of convergent Laurent series rings
on real tori.
It can be regarded as an Archimedean analog
of the construction of quantum analytic tori in \cite{MR2596639}.
In Section \ref{sc:root of unity},
we discuss a dual pair $X \to Y$
and $\hat{Y} \to \hat{X}$
of finite coverings of tori
associated with a deformation parameter $\lambda$
with values in roots of unity.
In Section \ref{sc:NC-finite},
we describe non-commutative complex tori
whose deformation parameters take values in roots of unity
in terms of a finite sheaf of algebras.
In Section \ref{sc:gerby},
we collect basic definitions
on twisted sheaves on complex manifolds.
In Section \ref{sc:FM},
we introduce
Fourier--Mukai transforms
from gerby complex tori
to 
non-commutative complex tori
at roots of unity
and state Theorem \ref{th:main},
which is the main result in this paper.
In Section \ref{sc:group actions},
we recall basic definitions and results
on finite group actions
on abelian and derived categories,
and discuss twistings by group cocycles.
In Section \ref{sc:examples of group actions},
we discuss group actions on categories
appearing in our construction.
In Section \ref{sc:proof of the main thoerem},
we prove Theorem \ref{th:main}.

\subsection*{Notations and conventions}

We fix the complex number field $\mathbb{C}$ as the ground field.
All modules (resp.~actions) are
right modules (resp.~actions)
unless otherwise specified.
In contrast,
all group actions on categories are left actions.
The word `non-commutative' is synonymous with `not necessarily commutative'.

For a ringed space $\cX=(Z,\cO_{\cX}),$
we write the ringed space $(Z,\cO_{\cX^{\mathrm{op}}}\coloneqq (\cO_{\cX})^\mathrm{op})$ as $\cX^{\mathrm{op}}.$
The category of $\cO_{\cX}$-modules will be denoted by $\Mod\cX.$
For an element $a$ of a complex abelian Lie group $A,$ we write the right translation map $A\rightarrow A,x\mapsto xa$ as $R_a.$ 
For two or three complex manifolds 
$Z_1,Z_2$ 
or 
$Z_1,Z_2,Z_3,$
the projection to the first (resp.~second) component
$Z_1$ 
(resp.~$Z_2$) 
will be denoted by 
$p^{Z_1,Z_2}$ 
or 
$p^{Z_1,Z_2,Z_3}$ 
(resp.~$q^{Z_1,Z_2}$ 
or 
$q^{Z_1,Z_2,Z_3}$). 
\subsection*{Acknowledgment}

The author is deeply grateful to his advisor,
Kazushi Ueda,
for a lot of guidance, useful comments, and encouragement.
The author also thanks Masahiro Futaki for many useful questions,
one of which lead to the formula \eqref{product on points}.
Finally, the author has deep gratitude to his parents
for their various supports throughout his life.

\section{$q$-Weyl algebras} \label{sc:q-Weyl}

The category of $\cO_X$-modules is equivalent
to the category of $\Gamma$-equivariant $\cO_T$-modules,
and the category of $\Gamma$-equivariant $\bC[T]$-modules can be regarded
as a toy model for it.
The latter can be identified
with the category of modules
over the crossed product algebra
$
\bC[T] \rtimes \Gamma.
$

For a parameter 
$
\lambda=(\lambda_{i,j})_{1\le i<j\le g}
\in (\bCx)^{g(g-1)/2},
$
the \emph{$q$-Weyl algebra} is
the non-commutative deformation of $\bC[T]$
defined by
\begin{align}
W_{\lambda}
\coloneqq
\bC \langle t_1^\pm,t_2^\pm,\ldots,t_g^\pm \rangle
/\left(t_i t_j - \lambda_{i,j} t_j t_i \right)_{1\le i<j\le g},
\end{align}
which gives $\bC[T]$ if
$
\lambda = 1 \coloneqq (1, \ldots, 1).
$

Let
$
Q = (q_{i,j})_{i,j=1}^g
$
be a multiplicative period matrix,
i.e.~a matrix
so that 
$
\left\{ 
\gamma_j\coloneqq(q_{i,j})_{i=1}^g
\right\}_{j=1}^g
$
is a free generator of $\Gamma$.
The group $\Gamma$ acts
on the $q$-Weyl algebra by
\begin{align}
t_i \cdot \gamma_j
\coloneqq q_{i,j}^{-1}t_i,
\end{align}
which reduces to the natural action on $\bC[T]$
when $\lambda = 1$.
The crossed product algebra
$
W_\lambda \rtimes \Gamma
$
is isomorphic to
another $q$-Weyl algebra
\begin{multline} \label{eq:q-Weyl nc}
\cW_{\lambda,Q,1}
\coloneqq
\bC\langle
    t_1^\pm,t_2^\pm,\ldots,t_g^\pm,
    \gamma_1^{\pm},\gamma_2^\pm,\ldots,\gamma_g^\pm
\rangle \\
/(\{t_it_j-\lambda_{i,j} t_jt_i\}_{1\le i<j\le g},
\{t_i\gamma_j-q_{i,j}^{-1}\gamma_jt_i\}_{i,j=1}^g,
\{[\gamma_i,\gamma_j]\}_{i,j=1}^g)   
\end{multline}
generated by $2g$ elements.

The parameter $\lambda$
describing a non-commutative deformation of $X$
can be used to describe a gerby deformation
of the dual torus $\hat{X}$:

\begin{dfn}
A \emph{$\lambda$-twisted $\cO_{\hat{X}}$-module}
is a pair
$
(\cM, (\rho_{\hat{\gamma}})_{\hat{\gamma}\in\hat{T}})
$
consisting of an
$\cO_{\hat{\Gamma}}$-module 
$\cM$
and a family
$(\rho_{\hat{\gamma}})_{\hat{\gamma}\in\hat{T}}$
of morphisms 
$
\rho_{\hat{\gamma}} \colon
\cM\rightarrow R_{\hat{\gamma}}^*\cM
$
satisfying
$
\rho_{\hat{\gamma}_j}\circ\rho_{\hat{\gamma}_i}
=\lambda_{i,j}\rho_{\hat{\gamma}_i}\circ\rho_{\hat{\gamma}_j}
$
for all $i, j \in \{ 1, \ldots, g \}$
such that $i<j.$

\end{dfn}

As a toy model of $\lambda$-twisted $\cO_{\hat{X}}$-modules,
we consider modules over
the $q$-Weyl algebra
\begin{multline} \label{eq:q-Weyl gerby}
  \cW_{1,Q,\lambda}
  \coloneqq
  \bC\langle
    \hat{t}_1^\pm,\hat{t}_2^\pm,\ldots,\hat{t}_g^\pm,
    \hat{\gamma}_1^{\pm},\hat{\gamma}_2^\pm,\ldots,\hat{\gamma}_g^\pm
  \rangle \\
  /(\{[\hat{t}_i,\hat{t}_j]\}_{i,j=1}^g,
  \{\hat{t}_i\hat{\gamma}_j-q_{j,i}^{-1}\hat{\gamma}_j\hat{t}_i\}_{i,j},
  \{\hat{\gamma}_i\hat{\gamma}_j
  -\lambda_{i,j}\hat{\gamma}_j\hat{\gamma}_i\}_{1\le i<j\le g}),
\end{multline}
which contains the ring of regular functions
$\bC[\hat{\Gamma}]\coloneqq \bC[\hat{t}_1^\pm,\hat{t}_2^\pm,\ldots,\hat{t}_g^\pm]$.
Note that the roles of $t_i$ and $\gamma_i$ are interchanged
between \eqref{eq:q-Weyl nc}
and \eqref{eq:q-Weyl gerby}.

A toy model for the deformed Poincaré line bundle,
which should give the integral kernel
of the deformed Fourier--Mukai transform,
is the
$
\left( \cW_{1,Q,\lambda}\right)^\op
\otimes
\cW_{\lambda,Q,1}
$-module
$P_\lambda$
such that
\begin{enumerate}
\item
$
P_\lambda
=
\bC[\hat{\Gamma}]\otimes_{\bC}W_\lambda
$
as a
$
\left( \bC[\hat{\Gamma}] \right)^\op
\otimes
W_\lambda
$-module,
\item
actions of $\gamma_i$ are given by
\begin{align*}
\left(
\psi(\hat{t}_1,\hat{t}_2,\ldots,\hat{t}_g)
\otimes
\phi(t_1,t_2,\ldots,t_g)
\right)
\cdot\gamma_i
=\psi(\hat{t}_1,\hat{t}_2,\ldots,\hat{t}_g)\hat{t}_i^{-1}
\otimes
\phi(t_1q_{1,i}^{-1},t_2q_{2,i}^{-1},\ldots,t_gq_{g,i}^{-1}),
\end{align*}
\item
actions of $\hat{\gamma}_i$ are given by
\begin{align*}
\hat{\gamma}_i\cdot
\left(
\psi(\hat{t}_1,\hat{t}_2,\ldots,\hat{t}_g)
\otimes
\phi(t_1,t_2,\ldots,t_g)
\right)
=
\psi(\hat{t}_1q_{i,1},\hat{t}_2q_{i,2},
\ldots,\hat{t}_gq_{i,g})
\otimes
t_i\phi(t_1,t_2,\ldots,t_g).
\end{align*}
\end{enumerate}

Note that the action of $\hat{\gamma}_i$ satisfies 
$
\hat{\gamma}_i \hat{\gamma}_j
=\lambda_{i,j} \hat{\gamma}_j \hat{\gamma}_i
$
since $t_it_j=\lambda_{i,j}t_jt_i$ in $W_\lambda$.

While $q$-Weyl algebras are only toy models
and one has to work with analytic functions
(such as theta functions)
rather than regular functions,
they provide intuition
behind constructions in later sections.
The duality between non-commutative deformations and gerby deformations is clearly visible
in this toy model.
Note also that
if all components of $\lambda$ are roots of unity,
then $W_\lambda$ is finite over its center,
which is the ring of functions on a finite quotient of $T$.

\section{Fourier--Mukai transforms} \label{sc:commutative}

We identify $\cO_X$-modules
with $\Gamma$-equivariant $\cO_T$-modules.
An element
$
\hat{x} \in \hat{\Gamma}
$
determines a $\Gamma$-equivariant
$\cO_T$-module
$
\cL_{\hat{x}}
$
as the trivial 
$\cO_T$-module
equipped with the $\Gamma$-action
\begin{align}
\phi(x) \cdot \gamma
=
\phi(x\gamma^{-1})\hat{x}(\gamma)^{-1}
\end{align}
for
$
\gamma \in \Gamma.
$
Since any
$
t \in \hat{T}
$
gives an isomorphism
\begin{align}
\cL_{\hat{x}}
\simto
\cL_{\hat{x}t^{-1}}, \qquad
\phi(x) \mapsto t(x)\phi(x)
\end{align}
of $\Gamma$-equivariant $\cO_T$-modules,
the map
$
\hat{x} \mapsto \cL_{\hat{x}}
$
descends to a map
$
\hat{\Gamma}/\hat{T}
\simto
\hat{X}
\coloneqq
\Pic^0 X,
$
which is an isomorphism of groups
because of the isomorphisms
\begin{align}
\cL_{\hat{x}} \otimes \cL_{\hat{x}'}
\simto \cL_{\hat{x} \hat{x}'}, \qquad
\phi(x)\otimes\psi(x) \mapsto \phi(x)\psi(x).
\end{align}

The \emph{Poincaré line bundle} $\cP$
is the
$\Gamma\times\hat{T}^{\mathrm{op}}$-equivariant
$\cO_{T \times \hat{\Gamma}}$-module,
defined as the trivial
$\cO_{T \times \hat{\Gamma}}$-module
equipped with the \emph{left} $\hat{T}$-action
\begin{align}
t \cdot \phi(x,\hat{x})
=t(x)\phi(x,\hat{x} t)
\end{align}
and the right $\Gamma$-action
\begin{align}
\phi(x,\hat{x})\cdot \gamma
=\phi(x\gamma^{-1} ,\hat{x})\hat{x}(\gamma)^{-1}.
\end{align}
Let 
$\overline{s}$ 
be the map given by
\begin{align}
\overline{s}
\colon
\hat{\Gamma}\times\hat{\Gamma}\times T
\rightarrow
\hat{\Gamma}\times T, \qquad
(\hat{x}_1,\hat{x}_2,x)
\mapsto (\hat{x}_1\hat{x}_2,x).
\end{align}
The map 
$
s \colon \hat{X}\times\hat{X}\times X\rightarrow \hat{X}\times X
$
is defined similarly.
Since the isomorphism
\begin{align}
\overline{m}
\colon
(p^{\hat{\Gamma},\hat{\Gamma},T})^*\cP
\otimes 
(q^{\hat{\Gamma},\hat{\Gamma},T})^*\cP
\rightarrow
\overline{s}^*\cP, \qquad
\phi\otimes\phi' \mapsto \phi \phi'
\end{align}
of $\cO_{\hat{\Gamma} \times \hat{\Gamma} \times T}$-modules
is $\hat{T}^\mathrm{op} \times \hat{T}^\mathrm{op} \times \Gamma$-equivariant,
it defines an isomorphism 
\begin{align}
m
\colon 
(p^{\hat{X},\hat{X},X})^*\cP
\otimes
(q^{\hat{X},\hat{X},X})^*\cP
\simto
s^*\cP
\end{align}
os $\cO_{\hat{X} \times \hat{X} \times X}$-modules,
which restricts to an isomorphism
\begin{align}
m_{\hat{x}_1,\hat{x}_2}
\colon
\cL_{\hat{x}_1}
\otimes \cL_{\hat{x}_2}
\simto
\cL_{\hat{x}_1 \hat{x}_2}
\end{align}
of $\cO_X$-modules
on
$
\{\hat{x}_1\}\times\{\hat{x}_2\} \times X
$
where
$
\cL_{\hat{x}} \simeq\cP|_{X \times \{\hat{x}\}}
$
by definition.

The \emph{Fourier--Mukai transform}
is the integral functor
with the Poincaré line bundle as the integral kernel;
\begin{align}
  \operatorname{FM}_\cP\colon
  D^b(\hat{X}) \rightarrow D^b(X), \qquad
  \cM \mapsto
  \bR\pps{X}{\hat{X}}
  (\qpl{X}{\hat{X}}\cM
  \otimes_{\cO_{X\times\hat{X}}}
  \cP).
\end{align}
It sends the skyscraper sheaf $\cO_{\hat{x}}$ of a point $\hat{x} \in \hat{X}$
to the line bundle $\cL_{\hat{x}}$.

\begin{thm}[{\cite{MR607081}}] \label{equivalence}
The Fourier--Mukai transform
$\operatorname{FM}_\cP$
is an equivalence,
whose quasi-inverse is given by the integral functor
$\operatorname{FM}_{\cP^{-1}[g]}$ 
with $\cP^{-1}[g]$ as the integral kernel.
\end{thm}

To be more precise,
Mukai proved this theorem for abelian varieties.
A proof for complex tori can be found in \cite{MR2309993}.

\section{Non-commutative deformations of complex tori} \label{sc:Laurent}

Set
\begin{align}
\varpi \colon T \rightarrow \abs{T} \coloneqq (\bR^{>0})^g, \qquad
(x_1,x_2,\ldots,x_g) \mapsto (|x_1|,|x_2|,\ldots,|x_g|).
\end{align}
For a product
$
D = \prod_{i=1}^g (r_i,R_i)
$
of open intervals,
the ring
$\varpi_* \cO_T(D)$ 
consists of
Laurent series in $g$ variables
with radii of convergence 
$(r_i,R_i)$ for $1\le i\le g$. 

\begin{dfn}
A \emph{unitary deformation parameter}
is an element of
$
Z^2(\hat{T},U(1)),
$
i.e., a map 
$
\lambda \colon \hat{T}\times\hat{T}\rightarrow U(1)
$
satisfying 
\begin{align}\label{eq:cocycle}
\lambda(t_2,t_3)\lambda(t_1t_2,t_3)^{-1}
\lambda(t_1,t_2t_3)\lambda(t_1,t_2)^{-1}=1
\end{align}
for all $t_1, t_2, t_3 \in \hat{T}$.
\end{dfn}

Given a unitary deformation parameter $\lambda$,
the \emph{star product}
$\ast_\lambda$ 
on
$\varpi_* \cO_T(D)$ 
is defined by
\begin{align}\label{eq:star product}
\left(
  \sum_{t\in \hat{T}}a_{t}t
\right)
\ast_\lambda
\left(
  \sum_{t\in\hat{T}}b_tt
\right)
=\sum_{t\in\hat{T}}
\left(
  \sum_{t_1,t_2\in\hat{T},\,t_1t_2=t}
\lambda(t_1,t_2) a_{t_1}b_{t_2}
\right)
t,
\end{align}
which is easily seen to be associative by using \eqref{eq:cocycle}.
The convergence of the right hand side
follows from
\begin{align}
\left|\sum_{t_1,t_2\in\hat{T},\,t_1t_2=t}
\lambda(t_1,t_2) a_{t_1}b_{t_2}\right|
\le \sum_{t_1,t_2\in\hat{T},\,t_1t_2=t}
\abs{a_{t_1}}\abs{b_{t_2}}
\end{align}
and 
\begin{align}
\sum_{t\in\hat{T}}
\left(\sum_{t_1,t_2\in\hat{T},t_1t_2=t}
\abs{a_{t_1}}\abs{b_{t_2}}
\right)t=\left(\sum_{t\in \hat{T}}|a_{t}|t\right)
\left(\sum_{t\in\hat{T}}|b_t|t\right),
\end{align} 
which depends on the unitarity of $\lambda$.
The resulting sheaf of associative algebras on $\abs{T}$
will be denoted by
$
\cO_{\sfT_\lambda}
$
which turns $\abs{T}$ into a non-commutative ringed space
$
\sfT_\lambda \coloneqq (\abs{T}, \cO_{\sfT_\lambda}).
$

A cochain
$
\alpha \in Z^1(\hat{T}, U(1))
$
bounding
$
\lambda, \lambda'
\in Z^2(\hat{T},U(1))
$
is a map
$
\alpha \colon \hat{T}\rightarrow U(1)
$
satisfying
\begin{align}
\lambda'(t_1,t_2)
= \lambda(t_1,t_2)\alpha(t_1)\alpha(t_2)\alpha(t_1t_2)^{-1}.
\end{align}
It gives an isomorphism
\begin{align}\label{boundary}
  \sum_{t\in \hat{T}}a_{t}t
  \mapsto\sum_{t\in \hat{T}}
  \alpha(t)a_{t}t
\end{align}
of the ring of sections,
which ensures that the isomorphism class of
the sheaf
$\cO_{{\sfT_\lambda}}$
of associative algebras
depends only on the cohomology class
$
[\lambda] \in H^2(\hat{T}, U(1)).
$

The natural $T$-action on $T$
induces a $T$-action on $\abs{T}$,
which lifts to a $T$-action
on the ringed space $\sfT_\lambda$
in such a way that
the morphism
$
\rho_{a}
\colon 
\cO_{\sfT_\lambda}
\rightarrow 
(R_{\varpi(a)^{-1}})_*\cO_{\sfT_\lambda}
$
of sheaves of associative algebras
for $a \in T$
is given by
\begin{align}
\sum_{t\in \hat{T}}a_{t}t
\mapsto
\sum_{t\in \hat{T}}a_{t}t(a)^{-1}t.
\end{align}
The action of
$\Gamma$ 
on 
$\abs{T}$
is free since
$\varpi(\gamma)=1$
for $1 \ne \gamma \in \Gamma$
contradicts
the freeness or the properness of $\Gamma$-action on $T$.

\begin{dfn} \label{df:nc complex torus}
The \emph{non-commutative complex torus}
associated with a complex torus $X = T/\Gamma$
and a unitary deformation parameter
$\lambda \in Z^2(\hat{T},U(1))$ 
is the non-commutative ringed space 
$
\sfX_\lambda \coloneqq \sfT_\lambda / \Gamma.
$
\end{dfn}

Recall that
a sheaf $\cM$ of $\cO_\cX$-modules
on a ringed space 
$\cX = (X,\cO_\cX)$ 
is said to be \emph{coherent}
if
\begin{enumerate}
\item
$\cM$ is finitely generated,
i.e.,
for any point $x\in X$,
there exists an open neighborhood $U$ of $x$
and an epimorphism 
$\cO_\cX^{\oplus m}|_U
\rightarrow 
\cM|_U\rightarrow 0$
for some $m \in \bN$,
and
\item
for any open set $U$
and
any $m \in \bN$,
the kernel of
any morphism 
$
\cO_\cX^{\oplus m}|_U
\rightarrow 
\cM|_U$ 
is finitely generated.
\end{enumerate} 
The full subcategory of 
$\Mod \cX$
consisting of
coherent modules
will be denoted by 
$\coh \cX.$

\begin{lem} \label{lm:coherence}
  The sheaf
  $
  \cO_{\sfT_1}
  $
  is coherent.
\end{lem}

\begin{proof}
It is clear that $\cO_{\sfT_1}$ is finitely generated.
For any open subset $U$ of $\abs{T}$,
a morphism 
$
\alpha \colon
\cO_{\sfT_1}^{\oplus m}|_U
\rightarrow 
\cO_{\sfT_1}|_U$
is the same as a morphism
$
\tilde{\alpha} \colon
\cO_{T}^{\oplus m}|_{\varpi^{-1}(U)}
\to
\cO_{T}|_{\varpi^{-1}(U)}.
$
For any $x \in U$,
let $V \subset U$ be an open neighborhood of $x$
obtained as the product of intervals.
Then $\varpi^{-1}(V)$ is Stein
and hence there is an epimorphism $\tilde{\beta}$
from $\cO_{\varpi^{-1}(V)}^{\oplus m'}$
to
the kernel of
$\tilde{\alpha}|_{\varpi^{-1}(V)}$
for some integer $m'$.
We ensure that 
$\varpi_*$ is exact
(and hence the morphism 
$\beta\coloneqq\varpi_*\tilde{\beta}$
is also an epimorphism)
by applying \cite[Corollary 11.5.4]{MR1900941} for every fiber of $\varpi$.
This shows that $\ker \alpha$ is finitely generated,
and Lemma \ref{lm:coherence} is proved.
\end{proof}

Non-commutative complex tori 
at $\lambda = 1$
are usual complex tori:

\begin{prop}\label{pr:loceq}
The adjunction
$
\varpi^* \dashv \varpi_*
$
induces an equivalence
$
\coh T
\simeq
\coh \sfT_1.
$
\end{prop}

\begin{proof}
Oka coherence theorem implies that
an object of $\Mod T$ is coherent
if and only if it is finitely presented.
Similarly,
Lemma \ref{lm:coherence} implies that
an object of $\Mod \sfT_1$ is coherent
if and only if it is finitely presented.

The functors
$\varpi^*$
and
$\varpi_*$
induce
mutually inverse functors
on categories of finitely presented modules
since
\begin{enumerate}
\item
the functor 
$\varpi_*$
is exact
(and hence preserves cokernels in particular),
\item
the functor $\varpi^*$ preserves cokernels
since it is a left adjoint, and
\item
$\varpi_*$ 
and
$\varpi^*$
interchanges
$\cO_T$
and
$\cO_{\sfT_1}$.
\end{enumerate}
\end{proof}

\begin{cor}\label{cr:loceq}
The adjunction
$
(\varpi_X)^* \dashv (\varpi_X)_*
$
induces an equivalence
$
\coh X
\simeq
\coh \sfX_1,
$
where 
$
\varpi_X\colon X\rightarrow \abs{X}
$
is the map induced by 
$
\varpi.
$
\end{cor}

The category $\coh \cO_{{\sfX_\lambda}}$ is abelian
if Problem \ref{pb:coherence} below
has an affirmative answer:

\begin{prob}[Oka coherence for non-commutative tori]
 \label{pb:coherence}
Is $\cO_{{\sfT_\lambda}}$ coherent?
\end{prob}

Yet another problem
is a generalization to non-unitary deformation parameters,
which would be needed for the duality
with general gerby deformations.

\section{Deformation parameters at roots of unity} \label{sc:root of unity}

As one can see
(e.g.~by noting that $\hat{T}$-modules are equivalent to 
$\bZ[T]\cong\bZ[t_1^{\pm1},t_2^{\pm},\ldots,t_g^{\pm}]$-modules, and the trivial $\hat{T}$-modules $\bZ$ has the Koszul resolution associated to $t_1-1,t_2-1,\ldots,t_g-1$)
that the cohomology
$H^i(\hat{T}, A)$
with coefficients in an abelian group $A$
with the trivial $\hat{T}$-action
is isomorphic to
$
\Hom(\wedge^i\hat{T}/(\wedge^i\hat{T})_{\mathrm{tors}},A)
\cong A^{\binom{g}{i}}
$
(note that 
$(\wedge^i\hat{T})_{\mathrm{tors}}$ 
is generated by torsion elements 
$a \wedge a \wedge t_1 \wedge \cdots \wedge t_{i-2}$
($a,t_1,\ldots,t_{i-2}\in \hat{T}$)
of order $2$).
Let
$
\Lambda \in \Hom(\wedge^2 \hat{T}, U(1))
$
be the element corresponding to
the class
$
[\lambda] \in H^2(\hat{T}, U(1))
$
of a unitary deformation parameter
$
\lambda \in Z^2(\hat{T}, U(1)),
$
and set
\begin{align}
\hat{H} &\coloneqq
\left\{
t \in \hat{T} \relmid \Lambda(t \wedge t') = 1 \text{ for all } t'\in\hat{T}
\right\},
\end{align}
so that one has an exact sequence
\begin{align}
1 \rightarrow \hat{H} \rightarrow \hat{T} \rightarrow \hat{K} \rightarrow 1,
\end{align}
and
$
[\lambda] \in H^2(\hat{T}, U(1))
$
descends to an element of
$
H^2(\hat{K}, U(1)),
$
which can be represented by a bilinear cochain
in
$
Z^2(\hat{K}, U(1)).
$

For the rest of this paper
and
unless otherwise specified,
we will
assume that a unitary deformation parameter
$\lambda \in Z^2(\hat{T},U(1))$
is contained in
$Z^2(\hat{T},\bsmu_N)$
for some positive integer $N$
where
$
\bsmu_N
\coloneqq
\left\{
\zeta \in U(1) \relmid \zeta^N = 1
\right\}.
$
This implies that $\hat{K}$ is a finite abelian group,
and
we will also fix
a bilinear map
\begin{align}
\lambda \colon \hat{K} \otimes_\bZ \hat{K} \to \bsmu_N
\end{align}
representing
$
[\lambda].
$

The dual group
$K \coloneqq \Hom(\hat{K}, \bCx)$
can be identified with the kernel of the map
$
T \to H
$
dual to the inclusion $\hat{H} \to \hat{T}$,
so that one has an exact sequence
\begin{align}
1 \rightarrow K \rightarrow T \xrightarrow{\pi_T} H \rightarrow 1.
\end{align}
Since $\Gamma \cap K$ is the trivial group,
the free action of $K$ on $T$ descends to a free action of $K$ on $X \coloneqq T/\Gamma$,
so that one has an exact sequence
\begin{align}
1 \rightarrow K \rightarrow X \xrightarrow{\pi} Y \rightarrow 1
\end{align}
where $Y$ is a complex torus.
We also have an exact sequence
\begin{align}
1 \rightarrow \hat{K} \rightarrow \hat{Y} \xrightarrow{\hat{\pi}} \hat{X} \rightarrow 1
\end{align}

Since $K$ goes to the identity
under the homomorphism
$
\varpi \colon T \to \abs{T},
$
there exists
$
\varpi_Y \colon Y \rightarrow \abs{X}
$
making the diagram
\begin{align}
  \begin{diagram}
    \node{X}\arrow{e,t}{\pi}
    \arrow{se,b}{\varpi_X}
    \node{Y}
    \arrow{s,r}{\varpi_Y}\\
    \node{}\node{\abs{X}}
  \end{diagram}
\end{align}
commute.

\section{Non-commutative deformations at roots of unity} \label{sc:NC-finite}

We have an isomorphism
\begin{align} \label{eq:varpi_* O_X}
\pi_* \cO_X
\cong
\bigoplus_{\hat{k} \in \hat{K}} \cL_{\hat{k}}
\end{align}
of sheaves of $\cO_Y$-algebras on 
$\hat{Y} \coloneqq \Pic^0 Y$.
The summand 
$
\cL_{\hat{k}}
$
is locally generated by a monomial function
$
t\in\hat{T}
$
representing 
$
\hat{k}.
$
We define a star product 
$\star_\lambda$
on 
$
\pi_* \cO_X
$
by
\begin{align}
\star_\lambda \colon
\bigoplus_{\hat{k} \in \hat{K}} \cL_{\hat{k}}
\times
\bigoplus_{\hat{k} \in \hat{K}} \cL_{\hat{k}}
&\rightarrow
\bigoplus_{\hat{k} \in \hat{K}} \cL_{\hat{k}} \\
((\phi_{\hat{k}})_{\hat{k}},(\psi_{\hat{k}})_{\hat{k}})
&\mapsto
\left(
\sum_{{\hat{k}}_1{\hat{k}}_2={\hat{k}}}
\lambda({\hat{k}}_1,{\hat{k}}_2)m_{{\hat{k}}_1,{\hat{k}}_2}
(\phi_{{\hat{k}}_1}\otimes\psi_{{\hat{k}}_2})
\right)_{\hat{k}}.
\label{star product}
\end{align}

We write the resulting sheaf
$
\left(
\pi_* \cO_X, \star_\lambda
\right)
$
of $\cO_Y$-algebras
as $\cO_{X_\lambda}$,
and
the ringed space
$
\left(
Y,
\cO_{X_\lambda}
\right)
$
as $X_\lambda$.

\begin{prop} \label{pr:Laurent and finite}
One has an isomorphism
$(\varpi_Y)_* \cO_{X_\lambda}
\cong \cO_{\sfX_\lambda}$
of sheaves of algebras.
\end{prop}

\begin{proof}
A direct calculation shows that
\begin{align}
&\left(
\sum_{t_1\in \hat{T}}a_{t_1}t_1
\right)
\ast_\lambda
\left(
\sum_{t_2\in\hat{T}}b_{t_2}t_2
\right)\\
&=
\left(
\sum_{t\in\hat{T}/\hat{H}}
\left(  
\sum_{t_1\bmod\hat{H}=t}
a_{t_1}t_1
\right)
\right)
\ast_\lambda
\left(
\sum_{t'\in\hat{T}/\hat{H}}
\left(  
\sum_{t_2\bmod\hat{H}=t'}
b_{t_2}t_2
\right)  
\right)\\
&=\sum_{t\in\hat{T}/\hat{H}}
\left(
\sum_{t'\in\hat{T}/\hat{H}}
\left(
\sum_{t_1\bmod\hat{H}=t}
\left(
\sum_{t_2\bmod\hat{H}=t'}
\lambda(t_1,t_2) a_{t_1}b_{t_2}t_1t_2
\right)
\right)
\right)\\
&=\sum_{t\in\hat{T}/\hat{H}}
\left(
\sum_{t'\in\hat{T}/\hat{H}}
\left(
\sum_{t_1\bmod\hat{H}=t}
\left(
\sum_{t_2\bmod\hat{H}=t'}
\lambda(t,t') a_{t_1}b_{t_2}t_1t_2
\right)
\right)
\right)\\
&=\sum_{t\in\hat{T}/\hat{H}}
\left(
\sum_{t'\in\hat{T}/\hat{H}}
\lambda(t,t')
m_{t,t'}
\left(
\sum_{t_1\bmod\hat{H}=t}
a_{t_1}t_1,
\sum_{t_2\bmod\hat{H}=t'}
b_{t_2}t_2
\right)
\right)
\end{align}
coincide with 
$\star_\lambda$.
\end{proof}

\begin{prop}
  The adjunction 
  $(\varpi_Y)^* \dashv (\varpi_Y)_*$
  induces an equivalence 
  $\coh \Xl \simeq \coh \sfX_\lambda.$
\end{prop}

\begin{proof}
We write the full subcategory of  
$\Mod \Xl$
(resp.~$\Mod \sfX_\lambda$)
consisting of coherent 
$\cO_Y$-modules
(resp.~$\cO_{\sfY_1}$-modules)
as
$(\coh Y)^{\hat{K},\lambda}$
(resp.~$(\coh \sfY_1)^{\hat{K},\lambda}$). 
Since $\cO_{\Xl}$
is a finite 
$\cO_Y$-algebra
by definition and 
$\cO_{\sfX_\lambda}$
is a finite 
$\cO_{\sfY_1}$-algebra 
by Proposition \ref{pr:Laurent and finite},
the category 
$\coh \Xl$
(resp.~$\coh \sfX_\lambda$)
is equivalent to 
$(\coh Y)^{\hat{K},\lambda}$
(resp.~$(\coh \sfY_1)^{\hat{K},\lambda}$).
It follows from Corollary
\ref{cr:loceq}
and Proposition
\ref{pr:Laurent and finite}
that
\linebreak
$\cO_{\sfX_\lambda}\in\ob{\coh\mathsf{Y}_1}$
and the adjunction 
$(\varpi_Y)^*\dashv (\varpi_Y)_*$
induces an equivalence 
$
(\coh Y)^{\hat{K},\lambda}
\simeq
(\coh \sfY_1)^{\hat{K},\lambda}.
$
\end{proof}

In particular, 
$\cO_{\sfX_\lambda}$
is a coherent sheaf on 
$\sfX_\lambda$
and hence
$\coh\cO_{\sfX_\lambda}$
is an abelian category.
\begin{rem}
It follows from the definition of $\hat{K}$
that there exists an isomorphism
$
\lambda^\sharp \colon \hat{K} \rightarrow K
$
such that 
$
\lambda(\hat{k}_1,\hat{k}_2)
=
\hat{k}_2(\lambda^\sharp(\hat{k}_1)).
$
If we write the inverse of 
$\lambda^\sharp$ 
as
$\lambda_\flat$ 
and
define
$
\omega\colon K\otimes_\mathbb Z K\rightarrow U(1)
$
by 
$
\omega(k_1,k_2)=\lambda(\lambda_\flat(k_1),\lambda_\flat(k_2)),
$
then the star product on $\cO_{X_\lambda}$ can alternatively be described as
\begin{align}\label{product on points}
  \phi(x)\star_\lambda\psi(x)
  =\frac{1}{\sharp K}\disum_{k_1,k_2\in K}
  \omega(k_1,k_2)\phi(xk_1)\psi(xk_2).
\end{align}
\end{rem}

A coherent sheaf on $X_\lambda$ consists of
an $\cO_Y$-module 
$\cM$
and
a collection of morphisms
$\{m_{\hat{k}}\colon
\cM\otimes\cL_{\hat{k}}\rightarrow\cM\}_{\hat{k}\in\hat{K}}$ 
such that the diagram
\begin{align}
  \begin{diagram}
    \node{\cM\otimes\cL_{\hat{k}_1}\otimes\cL_{\hat{k}_2}}
    \arrow{e,t}{m_{\hat{k}_1}\otimes\id}
    \arrow{s,t}{\lambda(\hat{k}_1,\hat{k}_2)\id\otimes m_{\hat{k}_1\hat{k}_2}}
    \node{\cM\otimes\cL_{\hat{k}_2}}
    \arrow{s,r}{m_{\hat{k}_2}}\\
    \node{\cM\otimes\cL_{\hat{k}_1\hat{k}_2}}
    \arrow{e,t}{m_{\hat{k}_1\hat{k}_2}}
    \node{\cM}
  \end{diagram}
\end{align}
commutes.
The dual $\cO_Y$-module
$
\cM^{-1} \coloneqq \cHom_{\cO_Y}(\cM, \cO_Y)
$
equipped with the transposes of $m_t$
gives a coherent sheaf on $X_\lambda^\op$.

\section{Twisted sheaves} \label{sc:gerby}

Let
$M$ 
be a complex manifold
and 
$\alpha\in H^2(M,\cO_M^*)$ 
be a second \'{e}tale cohomology class of 
$\cO_M^*$.
Take an \'{e}tale covering 
$\cU=(U_i)_i$ 
and a representative 
$(\alpha_{i,j,k})$ 
of 
$\alpha$ 
on 
$\cU$.
We write the  projections
from
$U_i \times_M U_j$ 
to the first (resp.~second) component as
$
I_{i,j}
$
(resp.~$J_{i,j}$).

\begin{dfn}
An
\emph{$\alpha$-twisted sheaf}
on
$M$ 
is a collection
$
((\cF_i)_i,
(\rho_{i,j})_{i,j})
$ 
of
$\cO_{U_i}$-modules 
$\cF_i$
and isomorphisms
$
\rho_{i,j}
\colon 
I_{i,j}^*\cF_i
\rightarrow
J_{i,j}^*\cF_j$ 
satisfying 
$
\rho_{i,k}^{-1}\circ \rho_{j,k}\circ \rho_{i,j}
=\alpha_{i,j,k}\operatorname{id}.
$
An $\alpha$-twisted sheaf is \emph{coherent}
if all $\cF_i$ are coherent.
\end{dfn}

The category of $\alpha$-twisted sheaves on 
$M$
and the full subcategory consisting of $\alpha$-twisted coherent sheaves
will be denoted by
$\Mod M^\alpha$ 
and
$\coh M^\alpha$
respectively.
Note that for an $\cO_M$-algebra $\cA$ and the resulting ringed space $\cX\coloneqq (M,\cA)$, we similarly can define the notion of $\alpha$-twisted sheaves on $\cX$ and categories $\Mod\cX^\alpha,\coh\cX^\alpha$. 
They do not depend
on the choice of $\cU$ and $(\alpha_{i,j,k})$
up to equivalence.
One has the tensor product functor
\begin{align}
\otimes
\colon
\Mod M^\alpha \times \Mod M^{\alpha'}
&\to
\Mod M^{\alpha \alpha'}, \\
\left(
(
(\cF_i)_i,
(\rho_{i,j})_{i,j}
),
(
(\cF_i')_i,
(\rho_{i,j}')_{i,j}
)
\right)
&\mapsto
(
(\cF_i \otimes \cF_i')_i,
(\rho_{i,j} \otimes \rho_{i,j}')_{i,j}
)
\end{align}
and the duality functor 
\begin{align}
(-)^{-1}
\colon
\Mod M^\alpha
&\to
\Mod M^{\alpha^{-1}}, \\
(
(\cF_i)_i,
(\rho_{i,j})_{i,j}
)
&\mapsto
(
(\cHom_{\cO_{U_i}}(\cF_i, \cO_{U_i}))_i,
((\rho_{i,j}^{-1})^{*})_{i,j}
).
\end{align}

%group cohomology and twisted sheaves  
If $\cU$ consists of
a principal $G$-bundle $P$
on $M$ 
for some discrete group $G$,
then the isomorphism
\begin{align}
P \times G^p
\rightarrow
P\times_M P\times_M \cdots \times_M P, \quad
(y,g_1,\ldots,g_p) \mapsto (y,yg_1,yg_1g_2,\ldots,yg_1\cdots g_p)
\end{align}
induces an isomorphism
from the \v{C}ech complex
$C^\bullet(\cU,\cO_M^*)$
to the standard complex
$C^\bullet(G,\cO_P^*(P))$
for group cohomology.
The composite of the resulting map
$
H^2(G,\cO_P^*(P))
\rightarrow
H^2(M,\cO_M^*)
$
with the map
$
H^2(G,\bCx)
\rightarrow 
H^2(G,\cO_P^*(P))
$
will be denoted by
$
\iota_P
\colon
H^2(G,\bCx)
\rightarrow
H^2(M,\cO_M^*).
$
For any
$
\lambda \in H^2(G, \bCx),
$
an $\iota_P(\lambda)$-twisted sheaf 
will simply be called a $\lambda$-twisted sheaf.
If 
$G$
is a finite abelian group, then 
$\iota_P(\lambda)$
is a torsion element since any group cohomology of finite abelian group is torsion.
In other words, $\iota_P(\lambda)$
is an element of the \emph{cohomological Brauer group}
$\operatorname{Br}(M)\coloneqq H^2(M,\mathcal O_M^*)_{\mathrm{tors}}$. 
A $\lambda$-twisted sheaf consists of an $\cO_P$-module 
$\cF$ and
a \emph{$\lambda$-twisted $G$-linearization} of $\cF$,
i.e.,
a collection
$
(\rho_g)_{g \in G}
$
of morphisms 
$
\rho_g\colon\cF\rightarrow R_g^*\cF
$ 
satisfying 
$R_{g_1}^*\rho_{g_2}\circ\rho_{g_1}=\lambda(g_1,g_2)\rho_{g_1g_2}$.

A $\lambda$-twisted sheaf
$(\cF, (\rho_g)_{g \in G})$ will also be called
a \emph{$\lambda$-twisted $G$-equivariant 
$\cO_P$-module};
it reduces to a $G$-equivariant $\cO_P$-module
if $\lambda = 1$
(which in turn is equivalent to an $\cO_M$-module).

\section{Deformed Fourier--Mukai transforms} \label{sc:FM}

Let 
$\cQ$ 
be the Poincaré line bundle on 
$Y\times \hat{Y}.$
For a complex manifold $Z$,
a sheaf of associative algebras
$\ppl{Y}{Z}\cO_{\Xl}$
(resp.~$\ppl{Y}{Z}\cO_{\Xl}^\mathrm{op}$)
will denoted by 
$\cO_{\Xl\times Z}$ 
(resp.~$\cO_{\Xl^\mathrm{op}\times Z}$),
and the resulting ringed space will be denoted by
$\Xl\times Z$
(resp.~$\Xl^\mathrm{op}\times Z$).
Symbols
$Y\times \hat{X}^\lambda$
(resp.~$\Xl\times\hXl,\Xl^\mathrm{op}\times\hXl$)
denote
$(Y\times \hat{X})^{1\times\lambda}$
(resp.~$(\Xl\times \hat{X})^{1\times\lambda},(\Xl^\mathrm{op}\times \hat{X})^{1\times\lambda}$).

The \emph{deformed Poincaré line bundle}
$
\cPl
$
is an object of $\coh \Xl \times \hXli$
defined as the $\cO_{Y\times \hat{Y}}$-module
\begin{align}
  \bigoplus_{\hat{k}\in\hat{K}}
  \cQ\otimes \ppl{Y}{\hat{Y}}
  \cL_{\hat{k}}
  \cong \bigoplus_{\hat{k}\in\hat{K}}
  R_{\hat{k}}^*\cQ
\end{align}
equipped with the 
$\cO_{\Xl \times\hat{Y}}
\cong
\bigoplus_{\hat{k}\in\hat{K}}
\ppl{Y}{\hat{Y}}\cL_{\hat{k}}
$-action
\begin{align}
  \bigoplus_{\hat{k}\in\hat{K}}
  \cQ\otimes \ppl{Y}{\hat{Y}}\cL_{\hat{k}}
  \times
  \bigoplus_{\hat{k}\in\hat{K}}
  \ppl{Y}{\hat{Y}}\cL_{\hat{k}}
  \rightarrow&
  \bigoplus_{\hat{k}\in\hat{K}}
  \cQ\otimes \ppl{Y}{\hat{Y}}\cL_{\hat{k}}\\
  ((\psi_{\hat{k}}\otimes\phi_{\hat{k}})_{\hat{k}},
  (\phi'_{\hat{k}})_{\hat{k}})
  \mapsto &
  \left(\sum_{\hat{k}_1,\hat{k}_2\in\hat{K},\,\hat{k}_1\hat{k}_2=\hat{k}}
  \psi_{\hat{k}_1}\otimes(\phi_{\hat{k}_1}
  \star_{\lambda}
  \phi'_{\hat{k}_2})\right)_{\hat{k}},
\end{align}
and the
$\lambda^{-1}$-twisted 
$\hat{K}$-action (i.e. the $\lambda$-twisted \emph{left}$\hat{K}$-action)
\begin{align}
  \rho_{\hat{k}}\colon
  \bigoplus_{\hat{k}'\in\hat{K}}
  R_{\hat{k}'}^*\cQ
  \rightarrow& 
  R_{\hat{k}^{-1}}^*
  \bigoplus_{\hat{k}'\in\hat{K}}
  R_{\hat{k}'}^*\cQ\\
  (\phi_{\hat{k}'})_{\hat{k}'}\mapsto&
    (\lambda(\hat{k},\hat{k}'\hat{k}^{-1})
    \phi_{\hat{k}'\hat{k}^{-1}})_{\hat{k}'}.
\end{align}

The \emph{deformed Fourier--Mukai transform}
\begin{align}
\FMl
\colon
D^b(\hXl)
\to
D^b(\Xl)
\end{align}
is the integral functor
with the deformed Poincaré line bundle as the integral kernel,
i.e., the composite of the pull-back
\begin{align}
\qpl{Y}{\hat{Y}}
\colon
D^b(\hXl)
\to
D^b(Y \times \hXl),
\end{align}
the tensor product
\begin{align}
(-)\otimes \cPl
\colon
D^b(Y\times\hXl)
\to
D^b(\Xl \times \hat{X}),
\end{align}
and the push-forward
\begin{align}
\bR\pps{Y}{\hat{X}}
\colon
D^b(\Xl \times \hat{X})
\to
D^b(\Xl).
\end{align}
Its right adjoint is the integral functor
$\FM^{-1}_\lambda$
with the $g$-shift of 
\begin{align}\label{eq:inverse kernel}
\cPl^{-1}
\coloneqq 
\cHom_{\cO_{Y\times \hat{Y}}}
(\mathcal P_\lambda,\cO_{Y\times \hat{Y}})
\in\ob{\coh\Xl^{\mathrm{op}}\times \hXl}
\end{align}
as the kernel,
since
\begin{itemize}
\item
the push-forward
$
\bR(q^{Y,\hat{Y}})_*
$
is right adjoint
to the pull-back
$
\qpl{Y}{\hat{Y}},
$
\item
the tensor product
$
\cPl^{-1} \otimes (-)
$
is right adjoint to the tensor product
$
(-) \otimes \cPl,
$
and
\item
the pull-back
$
(p^{Y, \hat{X}})^*[g]
$
shifted by $g$
is right adjoint to the push-forward
$
\bR\pps{Y}{\hat{X}}
$
because $D^b(X_\lambda)$ is Calabi--Yau of dimension $g$ and $D^b(X_\lambda\times \hat{X})$ is Calabi--Yau of dimension $2g$ (it will be proved in Section \ref{sc:examples of group actions}).
%We will use this adjunction only in Section \ref{sc:proof of the main thoerem}.
\end{itemize}

Theorem \ref{th:main} below is the main result in this paper:

\begin{thm} \label{th:main}
The deformed Fourier--Mukai transform $\FMl$ 
is an equivalence of derived categories.
\end{thm}

\begin{rem}
Let   
$\cO_{{\sfT_\lambda} \times \hat{\Gamma}}$
be the sheaf associative algebras 
$(\varpi\times\id)_*\cO_{T\times\hat{\Gamma}},$
with a non-commutative associative product defined by the formula similar to \eqref{eq:star product}, but 
$a_t$ and $b_t$ are functions on $\hat{\Gamma}.$
For
any $\lambda \in H^2(\hat{T},U(1))$
not necessarily at roots of unity,
it is natural to expect that
$\coh \hat{X}^\lambda$
is derived-equivalent to 
$\coh\sfX_\lambda$.
Although the latter is not known to be abelian,
we can define the deformed Poincaré line bundle 
$\sfP_\lambda$
as an object of
$
\Mod \left( \sfX_\lambda \times \hat{X}^{\lambda^{-1}} \right)
,$
i.e. a
$\cO_{{\sfT_\lambda} \times \hat{\Gamma}}$-module
equipped with a $\lambda$-twisted left $\hat{T}$-action and a right $\Gamma$-action.

$\sfP_\lambda$ is defined by a free 
$\cO_{{\sfT_\lambda} \times \hat{\Gamma}}$-module
of rank $1$ equipped with a $\lambda$-twisted left $\hat{T}$-action
\begin{align}
\hat{\gamma}\cdot\phi(x,\hat{x})
=\hat{\gamma}(x)\ast_\lambda\phi(x,\hat{x}\hat{\gamma})
\end{align}
and the right $\Gamma$-action
\begin{align}
\phi(x,\hat{x})\cdot \gamma
=\phi(x\gamma^{-1} ,\hat{x})\ast_\lambda\hat{x}(\gamma)^{-1}.
\end{align}
\end{rem}

\section{Finite group actions on abelian and derived categories} \label{sc:group actions}

It is natural to examine group actions on DG-categories in relation to group actions on derived categories and equivariant Fourier-Mukai transforms. 
However, coherent data for group actions on DG-categories are more intricate than those for group actions on abelian categories. 
As such, we will concentrate on group actions on abelian categories and the actions they induce on derived categories.

A \emph{weak action}
of a finite group $G$
on a category $\cC$
is a family
$
(g_*)_{g \in G}
$
of autoequivalences
$
g_* \colon \cC \to \cC
$
such that
the functor
$
(g_1)_* \circ (g_2)*
$
is isomorphic to $(g_1g_2)_*$
for any $g_1,g_2 \in G$.
An \emph{action}
is a weak action
equipped with a \emph{coherence data},
i.e.,
a family
$
(c_{g_1,g_2})_{g_1,g_2 \in G}
$
of isomorphisms
$
c_{g_1,g_2} \colon (g_1)_* \circ (g_2)_* \simto (g_1 g_2)_*
$ 
of functors
such that the diagram
\begin{align}
\begin{CD}
(g_1)_* \circ (g_2)_* \circ (g_3)_* 
@>{c_{g_1,g_2}}>> (g_1 g_2)_* \circ (g_3)_* \\
@V{c_{g_2,g_3}}VV @V{c_{g_1 g_2,g_3}}VV \\
(g_1)_* \circ (g_2 g_3)_* @>{c_{g_1,g_2 g_3}}>> (g_1 g_2 g_3)_*
\end{CD}
\end{align}
commutes for any $g_1, g_2, g_3 \in G$
(cf.~e.g.~\cite{MR1437497}).
An action is \emph{strict}
if the coherence data consists of identities.

Let $\cC$ be a category
equipped with an action of a finite group $G$.
The following definition is taken from \cite{MR3039826}:

\begin{dfn}[{\cite[Definition 3.1]{MR3039826}}]
A \emph{linearization}
of $A \in \ob{\cC}$
is a family
$
(\rho_g)_{g \in G}
$
of isomorphisms
$
\rho_g \colon A \simto g_*A
$
such that
the diagram
\begin{align}
  \begin{diagram}
    \node{A}\arrow{e,t}{\rho_{g_1}}
    \arrow{se,b}{\rho_{g_1g_2}}
    \node{(g_1)_*A}\arrow{s,r}
    {c_{g_1,g_2}\circ(g_1)_*(\rho_{g_2})}\\
    \node{}\node{(g_1g_2)_*A}
  \end{diagram}
\end{align}
commutes for any $g_1, g_2 \in G$.
An \emph{equivariant object}
is an object equipped with a linearization.
A \emph{morphism} of equivariant objects
from $(A,(\rho_g)_{g\in G})$
to $(A',(\rho'_g)_{g\in G})$
is a morphism
$
\phi \colon A\rightarrow A'
$
such that the diagram commute
\begin{align}
  \begin{CD}
    A@>{\phi}>>A'\\
    @V{\rho_g}VV @V{\rho'_g}VV\\
    g_*A@>{g_*\phi}>>g_*A'
  \end{CD}
\end{align}
commutes.
\end{dfn}

The category of $G$-equivariant objects in $\cC$
will be denoted by $\cC^G$.
For the rest of this paper and unless otherwise specified, we will assume that $\cC$ is a $\bC$-linear category and weak actions consists of $\bC$-linear functors.

\begin{prop}[{\cite[Proposition 3.2]{MR3039826}}] \label{pr:Sosna}
If $\cC$ is abelian,
then so is $\cC^G$.
\end{prop}

We extend the above constructions
to twisted group actions.
Let $\phi$ be a second cocycle of $G$
with values in $\bCx$.

\begin{dfn}
A 
\emph{$\phi$-twisted linearization}
of $A \in \ob{\cC}$
is a family
$
(\rho_g)_{g \in G}
$
of isomorphisms 
$
\rho_g \colon A \simto g_*A
$
such that 
the diagram
\begin{align}
  \begin{diagram}
    \node{A}\arrow{e,t}{\rho_{g_1}}
    \arrow{se,b}{\phi(g_1,g_2)\rho_{g_1g_2}}
    \node{(g_1)_*A}\arrow{s,r}
    {c_{g_1,g_2}\circ(g_1)_*(\rho_{g_2})}\\
    \node{}\node{(g_1g_2)_*A}
  \end{diagram}
\end{align}
commutes for any $g_1, g_2 \in G$.
Morphisms of
$\phi$-twisted equivariant objects
are defined in the same way
as in $\cC^G$.
\end{dfn}
A 
$\phi$-twisted linearization of 
$A\in\ob{\cC}$
is equivalent to a linearization of 
$A$
for $G$-action $\{\rho_g\}_{g\in G}$ equipped with a coherence data 
$(\phi(g_1,g_2)^{-1}c_{g_1,g_2})_{g_1,g_2\in G}$. 
The category of 
$\phi$-twisted equivariant objects
will be denoted by
$\cC^{G,\phi}.$ 
The cocycle condition on $\phi$ ensures the equality of
\begin{align}
  \rho_{g_3}\circ\rho_{g_2}\circ\rho_{g_1}
  =\rho_{g_3}\circ(\phi(g_1,g_2)\rho_{g_1g_2})
  =\phi(g_1,g_2)\phi(g_1g_2,g_3)\rho_{g_1g_2g_3}
\end{align}  
and 
\begin{align}
  \rho_{g_3}\circ\rho_{g_2}\circ\rho_{g_1}
  =(\phi(g_2,g_3)\rho_{g_2g_3})\circ\rho_{g_1}
  =\phi(g_2,g_3)\phi(g_1,g_2g_3)\rho_{g_1g_2g_3},
\end{align}
where we have omitted
$(g_1)_*$
and so on.
If
a pair of cocycles 
$\phi$ 
and
$\phi'$
differ by the coboundary of 
$
\alpha \in C^1(G,\bCx),
$
then there exists an equivalence
$
\cC^{G,\phi}
\to
\cC^{G,\phi'}
$
sending
$
(A,(\rho_g)_{g\in G})
$
to
$
(A,(\alpha(g)\rho_g)_{g\in G}).
$
Explanations of relations between group cohomology of $G$ in low degrees and (weak) $G$-actions are found in \cite{beckmann2020equivariant}.

Corollary \ref{cr:twisted equivariant categories} below is obtained by applying Proposition \ref{pr:Sosna} to the $G$-action equipped with a coherence data $(\phi(g_1,g_2)^{-1}c_{g_1,g_2})_{g_1,g_2\in G}$:

\begin{cor} \label{cr:twisted equivariant categories}
If $\cC$ is abelian,
then so is $\cC^{G, \phi}$.
\end{cor}

By applying the free-forgetful adjunction
\begin{align}
\free \dashv \fgt
\end{align}
between
\begin{align}
\free \colon
\cC
\to 
\cC^{G,\phi},
\qquad
A \mapsto
\left(
\bigoplus_{g' \in G} (g')_* A, \ 
\left( 
\rho_g \coloneqq \sum_{g' \in G} \phi(g,g')
\left(
\id\colon (gg')_*A\rightarrow g_*(g')_*A
\right)
\right)_{g \in G}
\right)
\end{align}
and
\begin{align}
\fgt
\colon
\cC^{G,\phi} \to \cC,
\qquad
(A,(\rho_g)_{g\in G}) \mapsto A
\end{align}
to the opposite categories
and using the equivalence
$
(\cC^{G,\phi})^\op \simeq (\cC^\op)^{G, \phi^{-1}},
$
one obtains an adjunction
\begin{align}
\fgt \dashv \free.
\end{align}

For any
$
(A,(\rho_g)_{g\in G}),
(A',(\rho'_g)_{g\in G})
\in
\cC^{G,\phi},
$ 
the space
$
\Hom_\cC(A,A')
$
comes with a natural linear action of $G$
in such a way that
the diagram
\begin{align}
  \begin{CD}
    A@>{\chi}>>A'\\
    @V{\rho_g}VV @V{\rho'_g}VV\\
    g_*A@>{g_*(\chi\cdot g)}>>g_*A'       
  \end{CD}
\end{align}
commutes
since
\begin{align}
\chi\cdot g_1\cdot g_2
&=((g_{1*})^{-1}(\rho'_{g_1}\circ\chi\circ \rho_{g_1}^{-1}))
\cdot g_2\\
&=(g_{2*})^{-1}
(\rho'_{g_2}\circ
((g_{1*})^{-1}
(\rho'_{g_1}\circ\chi\circ \rho_{g_1}^{-1}))
\circ \rho_{g_2}^{-1}))\\
&=(g_{2*})^{-1}(g_{1*})^{-1}
(g_{1*}\rho'_{g_2}
\circ(\rho'_{g_1}\circ\chi\circ \rho_{g_1}^{-1})
\circ g_{1*}\rho_{g_2}^{-1})\\
&=((g_1g_2)_*)^{-1}
((\phi(g_1,g_2)\rho'_{g_1g_2})
\circ\chi 
\circ (\phi(g_1,g_2)\rho_{g_1g_2})^{-1})\\
&=\chi\cdot g_1g_2.
\end{align}
It follows from the definition that
\begin{align}
\Hom_{\cC^G}
((A,(\rho_g)_{g\in G}),(A',(\rho'_g)_{g\in G}))
=
\Hom_\cC(A,A')^G.
\end{align}

A functor
$
\Phi \colon \cC \to \cC'
$
between categories with $G$-actions
is said to be
\emph{$G$-equivariant}
if it is equipped with a family
$
(a_g)_{g \in G}
$
of natural isomorphisms 
$
a_g \colon \Phi \circ g_* \simto g_* \circ \Phi
$
of functors
such that the diagram
\begin{align}\label{eq:pentagon}
\begin{diagram}
  \node{\Phi\circ (g_1)_*\circ (g_2)_*}
  \arrow{e,t}{a_{g_1}}\arrow{s,l}{c_{g_1,g_2}}
  \node{(g_1)_*\circ\Phi\circ (g_2)_*}
  \arrow{e,t}{a_{g_2}}
  \node{(g_1)_*\circ (g_2)_*\circ\Phi}
  \arrow{sw,b}{c_{g_1,g_2}}\\
  \node{\Phi\circ(g_1g_2)_*}
  \arrow{e,t}{a_{g_1g_2}}
  \node{(g_1g_2)_*\circ\Phi}
\end{diagram}  
\end{align}
commutes.
A $G$-equivariant functor
$
\Phi \colon \cC \to \cC'
$
induces a functor 
$
\Phi^{G,\phi} \colon
\cC^{G,\Phi} \to \cC'^{G,\phi}
$
sending an object 
$
(A,(\rho_g)_{g\in G})
$
to
$
(\Phi(A),(a_g\circ\Phi(\rho_g))_{g\in G})
$
and a morphism 
$
f \colon (A,(\rho_g)_{g\in G})
\rightarrow (A',(\rho'_g)_{g\in G})
$
to 
$
\Phi(f)\colon \Phi(A)\rightarrow \Phi(A').
$
It is straightforward to show that
$\Phi^{G,\phi}$ send $G$-equivariant objects to $G$-equivariant objects.

\begin{prop}
If $\Phi$ is right (resp.~left) exact,
then so is $\Phi^{G,\phi}$.
\end{prop}

\begin{proof}
Since $\free$ and $\fgt$ are mutually
both left and right adjoint to each other,
they are exact,
so that a sequence
$A\rightarrow B\rightarrow C$
in $\cC^{G,\phi}$ is exact
if and only if
$\fgt(A)\rightarrow \fgt(B)\rightarrow \fgt(C)$ is exact in $\cC$.
\end{proof}

Now we discuss the derived category of 
$\cC^{G,\phi}$.

\begin{prop}
An object
$(I,(\rho_g)_{g\in G})\in\ob{\tcg}$
is injective
if and only if 
so is $I \in \ob{\cC}$.
The category
$
\cC^{G,\phi}
$
has enough injectives if and only if 
so is $\cC$.
\end{prop}

\begin{proof}
If $I$ is injective,
then the functor
\begin{align}
  A
  \mapsto
  \Hom_{\cC^{G,\phi}}(A,(I,(\rho_g)))
  =\Hom_\cC(\fgt(A),I)^G
\end{align}
is exact, 
since $\fgt$ is exact,
$I$ is injective,
and taking the $G$-invariant part is exact.

Conversely,
if $(I,(\rho_g))$ is injective,
then the functor
\begin{align}
  A
  \mapsto
  \Hom_{\cC}(A,I)
  \cong \Hom_{\cC^{G,\phi}}
  (\free(A),(I,(\rho_g)))
\end{align}
is exact.

Let $(A,(\rho_g)_g)$ be an object in $\cC^{G,\phi}$. 
If $\cC$ has enough injectives,
then a monomorphism 
$A\rightarrow I$ 
into an injective object $I \in \ob{\cC}$
gives a monomorphism 
$
A \rightarrow \free I
$
into $\free I$,
which is injective in $\cC^{G,\phi}$. 

Conversely,
if 
$\tcg$ 
has enough injectives,
then
for any
$
A \in \ob{\cC},
$
a monomorphism 
$
\free A \rightarrow I
$
into an injective
$
I\in\ob{\tcg}
$
gives a monomorphism 
$
A \rightarrow \fgt I
$
into an injective $\fgt I$.
\end{proof}

Assume that
$\cC$ 
has enough injectives.
For any pair
$
A, B \in \ob{D^+(\tcg)}
$
of objects,
the space
$\Hom(\fgt(A),\fgt(B))$
has a natural $G$-action
in such a way that
\begin{align} \label{eq:RHom and G-invariants}
\Hom(A,B) \cong \Hom(\fgt(A),\fgt(B))^G.
\end{align}
 
\begin{prop}\label{pr:Calabi--Yau}
  If $D^b(\cC)$ is Calabi--Yau of dimension $n$, then so is 
  $D^b(\cC^{G,\phi})$.
\end{prop}
\begin{proof}
  There exists a natural isomorphism
  \begin{align}
    \Hom_{D^b(\cC)}(A,B)\cong\Hom_{D^b(\cC)}(B,A[n])^*
  \end{align} 
  for $A,B\in\ob{D^b(\cC^{G,\phi})}$, since $D^b(\cC)$ is Calabi--Yau of dimension $n$.
  It is $G$-equivariant by the naturality, so it induces an  isomorphism on $G$-invariant part.
  This means that $D^b(\cC^{G,\phi})$ is also Calabi--Yau of dimension $n$ by \eqref{eq:RHom and G-invariants}.
\end{proof}
A $G$-equivariant left exact functor
$
\Phi \colon \cC \to \cC'
$
induces functors
$
\bR \Phi \colon D^+(\cC)
\rightarrow D^+(\cC')
$
and
$
\bR \Phi^{G,\phi} \colon D^+(\tcg)
\rightarrow D^+(\cC'^{G,\phi}).
$
If
$
\bR \Phi
$
is fully faithful,
then
so is
$
\bR \Phi^{G,\phi}
$
by \eqref{eq:RHom and G-invariants}.

\section{Group actions on coherent sheaves} \label{sc:examples of group actions}

An action of a finite group $G$
on a complex manifold $Z$
induces a strict $G$-action
$(R_g^*)_{g \in G}$
on $\coh Z$.
In the case of the $\hat{K}$-action on $\hat{Y}$,
one obtains
$(\coh \hat{Y})^{\hat{K}, \lambda} \simeq \coh \hXl$.

Another example of a finite group action
on the category of coherent sheaves
comes from
a \emph{coherent} injection
from a finite abelian group $G$
to the Picard group $\Pic^0 Z$
of a complex manifold $Z$,
i.e.,
a family $\{ \cL_g \}_{g \in G}$ of line bundles
and a family $(m_{g_1, g_2})_{g_1,g_2 \in G}$
of isomorphisms
$
m_{g_1,g_2}
\colon
\cL_{g_1} \otimes \cL_{g_2}
\simto
\cL_{g_1g_2}
$
such that the diagrams
\begin{align}
\begin{CD}
  \cL_{g_1}\otimes \cL_{g_2}\otimes \cL_{g_3}
  @>{m_{g_1,g_2} \otimes \id}>> \cL_{g_1g_2}\otimes \cL_{g_3}\\
  @V{\id \otimes m_{g_2,g_3}}VV @V{m_{g_1g_2,g_3}}VV\\
  \cL_{g_1}\otimes \cL_{g_2g_3}@>{m_{g_1,g_2g_3}}>>\cL_{g_1g_2g_3}
\end{CD}
\end{align}
\begin{align}
\begin{diagram}\label{eq:symmetric}
  \node{\cL_{g_1}\otimes\cL_{g_2}}
  \arrow{e,t}{m_{g_1,g_2}}\arrow{s}
  \node{\cL_{g_1 g_2}}\\
  \node{\cL_{g_2}\otimes\cL_{g_1}}
  \arrow{ne,b}{m_{g_2,g_1}}
\end{diagram}
\end{align}
commutes,
inducing a $G$-action
$
\left(
(-) \otimes \cL_g^{-1})
\right)_{g \in G}
$
on
$
\coh Z.
$
Here,
the vertical arrow in \eqref{eq:symmetric} comes from the canonical symmetric monoidal structure in $\coh Z.$
If $Z$ is compact and connected,
then one has $\operatorname{Aut} \cL = \bCx$
for any line bundle $\cL$,
and an example of a coherence data 
$(m_{g_1,g_2})_{g_1,g_2 \in G}$
comes from a choice of a collection
$(\varphi_g)_{g \in G}$
of linear isomorphisms
$
\varphi \colon (\cL_g)_z \simto \bC
$
from the fibers $(\cL_g)_z$ of $\cL_g$
at an arbitrarily chosen and fixed base point $z \in Z$
to the complex line.

A $\phi$-twisted $G$-linearization
$(\rho_g)_{g \in G}$
of $\cM$
is equivalent to a family
$
(m_g)_{g \in G}
$
of morphisms
\begin{align}
m_g=(g_*)^{-1}\rho_g\colon
\cM\otimes \cL_g\rightarrow \cM
\end{align}
such that diagram
\begin{align}
\begin{diagram}
  \node{\cM\otimes \cL_{g_1}\otimes \cL_{g_2}}
  \arrow{e,t}{m_{g_1}}
  \arrow{se,b}{\phi(g_1,g_2)m_{g_1g_2}}
  \node{\cM\otimes \cL_{g_2}}\arrow{s,r}{m_{g_2}}\\
  \node{}\node{\cM}
\end{diagram}
\end{align}
commutes.
The category 
$(\coh Z)^{G,\phi}$
is equivalent to
$\coh \cA_\phi$,
where
$\cA_\phi$
is the sheaf of $\cO_Z$-algebras
obtained as the $\cO_Z$-module
$
\bigoplus_{g\in G} \cL_g
$
equipped with the multiplication
given by
$
\sum_{g_1,g_2\in G} \phi(g_1,g_2) m_{g_1,g_2}.
$
If 
$\phi=1$,
$\cA_\phi$
is commutative by the commutativity of \eqref{eq:symmetric}.
In particular, 
$\coh \Xl$
is equivalent to 
$
(\coh Y)^{\hat{K},\lambda}.
$
By Proposition \ref{pr:Calabi--Yau}, $D^b(X_\lambda)$ is Calabi--Yau of dimension $g$ and $D^b(X_\lambda\times\hat{X})$ is Calabi--Yau of dimension $2g$, which were needed to prove that $\FM_\lambda^{-1}$ is right adjoint to $\FM_\lambda$.

\section{Proof of Theorem \ref{th:main}} \label{sc:proof of the main thoerem}

Functors
$\FMl$
and
$\FMQ$
are the right derived functors of functors
\begin{align}
\FMla \colon \coh \hXl &\to \coh\Xl \\
\cM
&\mapsto
\pps{Y}{\hat{X}}(\qpl{Y}{\hat{Y}}\cM
  \otimes_{\cO_{Y\times\hat{Y}}}
  \cPl)
\end{align} 
and
\begin{align}
\FMQa \colon \coh \hat{Y} &\to \coh Y \\
\cM &\mapsto
\pps{Y}{\hat{Y}}(\qpl{Y}{\hat{Y}}\cM
  \otimes_{\cO_{Y\times\hat{Y}}}
  \cQ)
\end{align}
of abelian categories.

Since
\begin{align}
\FMQa
(R_{\hat{y}}^*\cM)
&= \pps{Y}{\hat{Y}}
(\qpl{Y}{\hat{Y}}R_{\hat{y}}^*\cM
\otimes\cQ)\\
&\cong \pps{Y}{\hat{Y}}
R_{(1,\hat{y}^{-1})}^*
(\qpl{Y}{\hat{Y}}R_{\hat{y}}^*\cM
\otimes\cQ)\\
&\cong \pps{Y}{\hat{Y}}
(\qpl{Y}{\hat{Y}}\cM\otimes
R_{(1,\hat{y}^{-1})}^*\cQ)\\
&\cong \pps{Y}{\hat{Y}}
(\qpl{Y}{\hat{Y}}\cM\otimes
\cQ\otimes\ppl{Y}{\hat{Y}}\cL_{\hat{y}}^{-1})\\
&\cong
\pps{Y}{\hat{Y}}(\qpl{Y}{\hat{Y}}\cM\otimes\cQ)
\otimes\cL_{\hat{y}}^{-1}\\
&= \operatorname{FM}^{\mathrm{ab}}_\cQ
(\cM)\otimes\cL_{\hat{y}}^{-1},
\end{align}
for any $\hat{y}\in\hat{K}$ and $\cM\in\ob{\coh\hat{Y}},$
$\FMQa$ commutes
with the weak $\hat{K}$-action.
The commutativity of the diagram \eqref{eq:pentagon} in this case is a straightforward diagram chasing.
This turns $\FMQa$ into a $\hat{K}$-equivariant functor,
inducing a functor
\begin{align}
(\FMQa)^{\hat{K},\lambda}
\colon
\coh \hXl
\to 
\coh \Xl.
\end{align}
\begin{lem}
  The functor 
  $(\FMQa)^{\hat{K},\lambda}$
  is isomorphic to 
  $\FMla.$
\end{lem}
\begin{proof}
The squares on the left and the right of the diagram 
\[
  \xymatrix@=54pt{
    (\operatorname{coh}\hat{Y})^{\hat{K},\lambda}
    \ar[r]^-{(\qpl{Y}{\hat{Y}})^{\hat{K},\lambda}}\ar[d]
    &(\operatorname{coh}(Y\times\hat{Y}))^{\hat{K},\lambda}
    \ar[r]^-{(\pi_*(-\otimes\mathcal Q))^{\hat{K},\lambda}}\ar[d]
    &(\operatorname{coh}(Y\times\hat{X}))^{\hat{K},\lambda}
    \ar[r]^-{(\pps{Y}{\hat{X}})^{\hat{K},\lambda}}\ar[d]
    &(\operatorname{coh}Y)^{\hat{K},\lambda}
    \ar[d]\\
    \operatorname{coh}\hXl
    \ar[r]^-{\qpl{Y}{\hat{Y}}}
    &\operatorname{coh}(Y\times\hXl)
    \ar[r]^-{-\otimes\cPl}
    &\operatorname{coh}(\Xl\times\hat{X})
    \ar[r]^-{\pps{Y}{\hat{X}}}
    &\operatorname{coh}\hXl
  }
\]
commute by definition, and the natural isomorphism
\begin{align}
  \displaystyle\oplus_{\hat{y}'\in\hat{K}}
  \rho_{\hat{y}'}^{-1}\otimes\operatorname{id}
  \colon
  \displaystyle\bigoplus_{\hat{y}'\in\hat{K}}
  R_{(1,\hat{y}')}^*
  (\mathcal M\otimes_{\mathcal O_{Y\times\hat{Y}}}\mathcal Q)
  \rightarrow
  \displaystyle\bigoplus_{\hat{y}'\in\hat{K}}    
  \mathcal M\otimes_{\mathcal O_{Y\times\hat{Y}}}
  R_{(1,\hat{y}')}^*\mathcal Q,
\end{align}
whose $\hat{K}$-equivariance can be checked
by a straightforward computation,
gives the commutativity of the square in the middle.
\end{proof}
Therefore
$\bR(\FMQa)^{\hat{K}, \lambda}$
and 
$\FMl$
are isomorphic.
Hence
$\FMl$
is fully faithful as explained at the end of Section \ref{sc:group actions}.
Similarly, the right adjoint 
$\FM^{-1}_\lambda$
of
$\FMl,$ whose kernel is given by the $g$-shift of \eqref{eq:inverse kernel}, is also fully faithful,
and
Theorem \ref{th:main} 
is proved.

\bibliographystyle{amsalpha}
\bibliography{my_reference_papers}

\end{document}